\documentclass{amsart}
\usepackage{amsmath,amsfonts,amsthm,amssymb,verbatim,enumerate,quotes,graphicx,mathtools}
\usepackage{url}
\newtheorem{theorem}{Theorem}
\newtheorem*{theoremA}{Theorem A}
\newtheorem*{theoremB}{Theorem B}

\newtheorem{lemma}{Lemma}

\newtheorem*{remark}{Remark}
\newtheorem*{remarks}{Remarks}
\newtheorem{question}{Question}

\def \isnatural {\in\mathbb{N}}
\def\R{\mathbb{R}}
\def\C{\mathbb{C}}
\def\N{\mathbb{N}}
\def\Z{\mathbb{Z}}

\def\D{\mathbb{D}}

\newcommand\qfor{\quad\text{for }}
\numberwithin{equation}{section}
\newcommand{\A}{{\mathcal A}}
\newcommand{\B}{{\mathcal B}}
\newcommand{\Bdisc}{{\mathcal{B}_{\D}}}
\newcommand{\LL}{{\mathcal L}}

\makeatletter
\def\blfootnote{\xdef\@thefnmark{}\@footnotetext}
\makeatother
\begin{document}
%
%
%
%
\title[The MacLane class and the Eremenko-Lyubich class]{The MacLane class and the Eremenko-Lyubich class}
\author{Karl F. Barth, \ Philip J. Rippon, \ David J. Sixsmith}
\address{Department of Mathematics \\
	 Syracuse University \\
   Syracuse NY 13244\\
   USA}
\address{Department of Mathematics and Statistics \\
	 The Open University \\
   Walton Hall\\
   Milton Keynes MK7 6AA\\
   UK}
\email{phil.rippon@open.ac.uk}
\address{Dept. of Mathematical Sciences \\
   University of Liverpool \\
   Liverpool L69 7ZL \\
   UK}
\email{david.sixsmith@open.ac.uk}
\dedicatory{Sadly, Professor Karl Barth died in May 2016, not long before the work on this article was completed. The other authors dedicate this paper to him.}
%
%
%
%
\begin{abstract}
In 1970 G. R. MacLane asked if it is possible for a locally univalent function in the class $\A$ to have an arc tract. This question remains open, but several results about it have been given. We significantly strengthen these results, in particular replacing the condition of local univalence by the more general condition that the set of critical values is bounded. Also, we adapt a recent powerful technique of C.~J.~Bishop in order to show that there is a function in the Eremenko-Lyubich class for the disc that is not in the class $\A$.
\end{abstract}
\maketitle
%
%
%
%
\section{Introduction}
We let $\D$ be the unit disk $\{z \in \C : |z| < 1\}$, and  denote the extended complex plane by $\widehat{\C} = \C \cup \{\infty\}$. We say that a path $\Gamma : [0,1) \to \D$ is a \emph{boundary path} if $|\Gamma(t)| \rightarrow 1$ as $t \rightarrow 1$. We call the set $\overline{\Gamma} \cap \partial \D$ the \emph{end} of $\Gamma$. We say that a holomorphic function $f$ defined in $\D$ \emph{has asymptotic value} $a \in \widehat{\C}$ if
there is a boundary path $\Gamma$ such that $f(\Gamma(t)) \rightarrow a$ as $t\rightarrow 1$. In this case we say that $\Gamma$ is an \emph{asymptotic curve over} the value $a$. If the end of $\Gamma$ is a singleton, $\{\zeta\}$, then we say that $f$ has asymptotic value $a$ \emph{at} $\zeta$. 

Denote the chordal metric on $\widehat{\C}$ by $d(a,b)$, and let 
$$
U(a,\epsilon) = \{z \in \widehat{\C} : d(a,z) < \epsilon\}, \qfor a \in \widehat{\C} \text{ and } \epsilon > 0.
$$
Suppose that $a \in \widehat{\C}$, and that to each $\epsilon > 0$ there is a component $D(\epsilon)$ of $f^{-1}(U(a,\epsilon))$ with the properties that $D(\epsilon_1 ) \subset D(\epsilon_2)$, for $0 < \epsilon_1 < \epsilon_2$, and $\bigcap_{\epsilon>0} D(\epsilon) = \emptyset$. Then we say that the pair $\{D(\epsilon),a\}$ is a \emph{tract} of $f$ \emph{over} the value~$a$. The set $K = \bigcap_{\epsilon>0} \overline{D(\epsilon)}$ is called the \emph{end} of the tract. It is easily shown that $K$ is either a point, in which case the tract is called a \emph{point} tract, or an arc, in which case the tract is called an \emph{arc} tract. If $K = \partial \D$ and for each arc $\gamma \subset \partial \D$ there is a sequence of arcs $\gamma_n \subset D(1/n)$ which tend to $\gamma$, then $K$ is called a \emph{global tract}.
%
%

The \emph{MacLane class}, $\A$, is defined as those functions holomorphic and non-constant in $\D$ that have asymptotic values at points of a dense subset of $\partial \D$; we refer to \cite{MR0148923} for an extensive discussion of this class. MacLane (\cite[p.281]{MR0274765}, and see also \cite[p.570]{MR0247100}), asked the following question about locally univalent functions in the class $\A$.
\begin{question}
\label{q1}
If $f \in \A$ is such that $f' \ne 0$, can $f$ have an arc tract?
\end{question}
After nearly half a century, this question remains open. However, many partial results have been obtained, and our goal in this paper is to generalise these. In particular, we seek to relax the condition of local univalence (in other words, $f' \ne 0$) to the condition that $CV(f)$ is bounded. Here $$CV(f) = \{ f(z) : z \in \D \text{ and } f'(z) = 0\}$$ denotes the set of critical values of $f$, in other words the images of the critical points of $f$. We also denote the set of \emph{finite} asymptotic values of $f$ by $AV(f)$.

%
%
We begin by recalling the following theorem, which combines the main results of \cite{MR1623386}. Here, following MacLane \cite{MR0274765}, we say that infinity is a \emph{linearly accessible} asymptotic value at a point $\zeta \in \partial \D$ if there is an asymptotic curve over infinity, $\Gamma$, which ends at $\zeta$, such that $f(\Gamma)$ is a radial line.
\begin{theoremA}
Suppose that $f \in \A$ and that $f' \ne 0$ in $\D$. If either of the following conditions also holds, then $f$ has no arc tracts.
\begin{enumerate}[(a)]
\item The set $AV(f)$ is bounded.
\item Infinity is a linearly accessible asymptotic value at a dense subset of $\partial \D$.
\end{enumerate}
\end{theoremA}

We generalise Theorem A in two ways. Firstly, we replace the condition of local univalence by the condition that $CV(f)$ is bounded. Secondly, we replace ``linearly accessible'' by the following weaker condition. We say that infinity is a \emph{monotonically accessible} asymptotic value at a point $\zeta \in \partial \D$ if there is an asymptotic curve over infinity, $\Gamma$, which ends at $\zeta$, such that $f(\Gamma)$ is monotonic; a curve $\gamma \subset \C$ is \emph{monotonic} if there exists $R> 0$ such that for all $r > R$ there is exactly one point on $\gamma$ of modulus $r$. 

Note that the proofs of both generalisations involve new ideas, which may be of wider interest; see Lemma~\ref{lemm:components}, for example.
\begin{theorem}
\label{theo:conditionsfornoarctracts}
Suppose that $f \in \A$ and that $CV(f)$ is bounded. If either of the following conditions also holds, then $f$ has no arc tracts.
\begin{enumerate}[(a)]
\item The set $AV(f)$ is bounded.\label{t1first}
\item Infinity is a monotonically accessible asymptotic value at a dense subset of $\partial \D$.\label{t1second}
\end{enumerate}
\end{theorem}

Another partial answer to Question~1 was given by MacLane \cite[p.281]{MR0274765}, who showed that if $f \in \A$ and $f' \ne 0$, then $f$ has no \emph{global} tracts. The following is a natural generalisation of this fact.
\begin{theorem}
\label{theo:noglobaltracts}
Suppose that $f \in \A$ and that $CV(f)$ is bounded. Then $f$ has no global tracts.
\end{theorem}

%
%

A further result about locally univalent functions generalised in this paper is the following theorem of the first two authors \cite[Theorem 1]{MR1874247}; see also \cite[Theorem 9]{MR0274765}.
\begin{theoremB}
Suppose that $f$ is holomorphic in $\D$ and that $f' \ne 0$ there. Then $f \in \A$ and $f$ has no arc tracts if and only if $f'\in \A$.
\end{theoremB}
As in Theorem~\ref{theo:conditionsfornoarctracts} and Theorem~\ref{theo:noglobaltracts}, our goal is to replace the condition of local univalence with the condition that $CV(f)$ is bounded. Our result is as follows.
\begin{theorem}
\label{theo:moreconditions}
Suppose that $f$ is holomorphic in $\D$ and that $CV(f)$ is bounded. Then the following both hold.
\begin{enumerate}[(a)]
\item If $f \in \A$ and $f$ has no arc tracts, then $f'\in \A$.
\item If $f' \in \A$ and $f'$ has no arc tracts, then $f\in \A$ and $f$ has no arc tracts.
\end{enumerate}
\end{theorem}

\begin{remark}\normalfont
It follows from \cite[Theorem 7]{MR0148923} that if $f' \ne 0$ in $\D$ and $f' \in \A$, then $f'$ has no arc tracts, so Theorem~\ref{theo:moreconditions} does indeed generalise Theorem~B. However, it is not clear that, in the statement of Theorem B, the condition that $f' \ne 0$ (which explicitly tells us that $f'$ omits a value) can be replaced by the condition that $CV(f)$ is bounded.
\end{remark}

%
%

In view of these results, it seems natural to ask the following stronger version of Question~\ref{q1}.
\begin{question}
\label{q2}
If $f \in \A$ is such that $CV(f)$ is bounded, can $f$ have an arc tract?
\end{question}

%
%

The final result in this paper concerns the relationship between the MacLane class and another well-known class of functions. In \cite{2015arXiv150200492R} the class of functions $f$ holomorphic and non-constant in a general hyperbolic domain, with the property that $CV(f) \cup AV(f)$ is bounded, was discussed and called the \emph{Eremenko-Lyubich class}. Transcendental entire functions with this property have been very widely studied, particularly in complex dynamics; see, for example, \cite{MR1196102, MR2570071, MR2753600, MR1357062}. Here we denote the Eremenko-Lyubich class in $\D$ by $\Bdisc$ to avoid confusion with MacLane's class $\B$, which we use later in the paper. 

The first part of Theorem~\ref{theo:conditionsfornoarctracts} states that if $f \in \A \cap \Bdisc$, then $f$ has no arc tracts. 
It is natural to ask, therefore, if there is any relationship between the classes $\A$ and $\Bdisc$. It is clear that $\A \not\subset \Bdisc$; for example, any conformal map of $\D$ onto an unbounded domain is in class $\A$ but is not in class $\Bdisc$. It is not so clear whether $\Bdisc \subset \A$, since no simple counterexample seems to be available. However, we are able to prove the following.
\begin{theorem}
\label{theo:finBnotA}
There is a function $f \in \Bdisc$ such that $f \notin \A$.
\end{theorem}

To prove this theorem, we first generalise a remarkable construction of Bishop \cite[Theorem 1.1]{MR3384512} from the complex plane to the disc. We then use this result to construct an unbounded function $f \in \Bdisc$ that is bounded on a spiral which accumulates on $\partial \D$. It then follows from \cite[Theorem 9]{MR0148923} that $f \notin \A$.

\begin{remarks}\normalfont
\mbox{ }
\begin{enumerate}
\item The use of Bishop's powerful method of construction seems to be essential here. Earlier examples of unbounded holomorphic functions in $\D$ that are bounded on a spiral were constructed using approximation theory. For example, using this approach, Barth and Schneider \cite{MR0247093} answered a question of Seidel by constructing a holomorphic function in $\D$ that is bounded on a spiral and whose only asymptotic value is infinity (but without any constraints on the critical points). On the other hand, Barth and Rippon \cite{MR1874247} constructed a holomorphic function in $\D$ that is bounded on a spiral with no critical points and asymptotic value infinity (but with no restrictions on the finite asymptotic values); see also \cite[Example 1]{MR2244003} for a similar example which satisfies $ff' \ne 0$ in $\D$.
\item We are grateful to Lasse Rempe-Gillen who, after this paper was completed, pointed out that a similar example to the one in Theorem~\ref{theo:finBnotA} was given by Bishop \cite[Section 14]{Bish3}. In fact, Bishop's example has no finite asymptotic values and only two singular values. Since Bishop gives very little detail, for the benefit of the reader we have retained an outline of the construction in the class $\Bdisc$.
\end{enumerate}
\end{remarks}
%
%

\subsection*{Structure}
The structure of the paper is as follows. First, in Section~\ref{S1}, we prove the first part of Theorem~\ref{theo:conditionsfornoarctracts} and then Theorem~\ref{theo:noglobaltracts}. In Section~\ref{S2} we prove the second part of Theorem~\ref{theo:conditionsfornoarctracts}, and in Section~\ref{S4} we prove Theorem~\ref{theo:moreconditions}. Finally, in Section~\ref{S3} and Section~\ref{S5} we prove Theorem~\ref{theo:finBnotA}.

%
%

\subsection*{Notation}
It is useful to define the right half-plane $\mathbb{H}_r = \{ z \in \C : \operatorname{Re}(z) > r \}$, for $r\in\R$. For simplicity we also write $\mathbb{H} = \mathbb{H}_0$. For $r>0$, we denote the open ball with centre at the origin by $B_r : = \{ z \in \C :|z| < r\}$, and the circle with centre at the origin by $C_r = \{ z \in \C : |z| = r\}$.

%
%

\subsection*{Note}
The definition of the class $\A$ is restricted to functions holomorphic in $\D$. It should be noted that this definition can be extended to a simply connected domain $U$ in an obvious way, and then we say that a function is \emph{in class $\A$ relative to $U$}. We also note that results about the class $\A$ in the disc can be generalised to any simply connected domain bounded by a Jordan curve via a Riemann map.
%
%
%
\section{Proof of the first part of Theorem~\ref{theo:conditionsfornoarctracts}, and Theorem~\ref{theo:noglobaltracts}}
\label{S1}
First in this section we give the proof of Theorem~\ref{theo:conditionsfornoarctracts} part (\ref{t1first}). In fact, we prove a slightly different result, and show that Theorem~\ref{theo:conditionsfornoarctracts} part (\ref{t1first}) is a consequence of this.

We use Iversen's classification of singularities, introduced in \cite{Iversen}. Suppose that $f$ is a function holomorphic in $\D$. We say that a tract $\{D(\epsilon),a\}$ is \emph{direct} if there exists $\epsilon > 0$ such that $f$ omits $a$ in $D(\epsilon)$; otherwise the tract is \emph{indirect}. If a tract $\{D(\epsilon),a\}$ is direct, and there exists $\epsilon > 0$ such that $f$ is a universal covering from $D(\epsilon)$ onto $U(a,\epsilon)\setminus\{a\}$, then we say the tract is \emph{direct logarithmic}; otherwise the tract is \emph{direct non-logarithmic}.
\begin{theorem}
\label{theo:directlog}
Suppose that $f \in \A$ and that $D = \{D(\epsilon), a\}$ is an arc tract of $f$. Then $D$ is a direct non-logarithmic tract over infinity.
\end{theorem}

Before proving Theorem~\ref{theo:directlog}, we show that Theorem~\ref{theo:conditionsfornoarctracts} part (\ref{t1first}) can be deduced easily from this result.

\begin{proof}[Proof of Theorem~\ref{theo:conditionsfornoarctracts} part (\ref{t1first})]
Suppose that $f\in\A$ and that $CV(f) \cup AV(f)$ is bounded. It follows from the fact that $CV(f) \cup AV(f)$ is bounded that there exists $\epsilon>0$ sufficiently small that $f$ is a universal covering from each component of $f^{-1}(U(\infty, \epsilon))$ onto $U(\infty, \epsilon)\setminus\{\infty\}$. Hence the only tracts of $f$ over infinity are direct logarithmic. The result then follows by Theorem~\ref{theo:directlog}.
\end{proof}

The proof of Theorem~\ref{theo:directlog} requires some auxiliary lemmas. The first combines some of MacLane's results \cite[Theorems~4~and~7]{MR0148923}, and is also used elsewhere in this paper. 
Here, if $K \subsetneq \partial \D$ is an arc, then we define a \emph{boundary neighbourhood} of $K$ as a domain $U = V \cap \D$ where $V$ is a simply connected domain such that $K$ is compactly contained in a crosscut of $V$.

\begin{lemma}
\label{lemm:maclane}
Suppose that $f \in \A$ has an arc tract $D = \{ D(\epsilon), a\}$ with end $K$. Then the following all hold.
\begin{enumerate}[(a)]
\item The tract $D$ is over infinity.\label{item:overinf} 
\item The function $f$ assumes every finite value infinitely many times in each boundary neighbourhood of each non-degenerate subarc of~$K$.\label{item:allvals}
\item Each set $D(\epsilon)$, $\epsilon > 0$, is infinitely connected.\label{item:infcon}
\end{enumerate}
\end{lemma} 

We note that it follows from Lemma~\ref{lemm:maclane} part \eqref{item:overinf} that any function in the MacLane class which has an arc tract is unbounded. We also use the following standard result; this can be proved in an almost identical way to, for example, \cite[Theorem 5.10]{MR1185074}, which covers the case $r = 0$.
\begin{lemma}
\label{lemm:covering}
Suppose that $W\subset\C$ is a domain, that $r \in [0, 1)$, and that $$g:W\to\{z\in\C : r < |z| < 1\}$$ is an unbranched covering map. Then exactly one of the following holds:
\begin{enumerate}[(a)]
\item there exists a conformal map $\psi: W \to \{ z \in \C : \log r < \operatorname{Re}(z) < 0\}$ such that $g~=~\exp~\circ~\psi$;
\item there exists a conformal map $\psi: W \to \{ z \in \C : r^{1/m} < |z| < 1 \}$ such that $g = (\psi)^m$, for some $m\isnatural$.
\end{enumerate}
\end{lemma}

\begin{proof}[Proof of Theorem~\ref{theo:directlog}]
Suppose that $f \in \A$ and that $f$ has an arc tract $D$ with end~$K$. It follows from Lemma~\ref{lemm:maclane} part \eqref{item:overinf} that $D = \{D(\epsilon), \infty\}$ is a tract over infinity. In particular, $D$ is a direct tract. It remains, therefore, to prove that $D$ is not a logarithmic tract. 

Suppose, by way of contradiction, that $D$ is a direct logarithmic tract. It follows that we can choose $\epsilon_0>0$ sufficiently small that 
$f : D(\epsilon_0) \to U(\infty, \epsilon_0)\setminus\{\infty\}$ is an unbranched covering map. 
Hence there exists $R_0\in\R$ such that the map $$g(z) = \frac{{e^{R_0}}}{f(z)}$$ is an unbranched covering map from $D(\epsilon_0)$ onto $\D\setminus\{0\}$, and so we can apply Lemma~\ref{lemm:covering} to $g$, with $W = D(\epsilon_0)$ and $r = 0$. Note that Lemma~\ref{lemm:covering} case (b) cannot hold, since it stands in contradiction to Lemma~\ref{lemm:maclane} part \eqref{item:infcon}. 
It follows that there is a conformal map $\phi : D(\epsilon_0) \to \mathbb{H}_{R_0}$ such that 
$$
f(z) = \exp(\phi(z)), \qfor z \in D(\epsilon_0).
$$

Choose $r > R_0$, and consider the set $T_r = \phi^{-1}({\mathbb{H}_r})$. Since there exists $\epsilon~\in~(0,~\epsilon_0)$ such that $f(T_r) = U(\infty, \epsilon)\setminus\{\infty\}$, we have that $T_r = D(\epsilon)$. We also have that $\phi^{-1}(\partial \mathbb{H}_r) = \partial T_r \cap \D$ is connected. This is also in contradiction to Lemma~\ref{lemm:maclane} part \eqref{item:infcon}, which completes the proof. 
\end{proof}

The proof of Theorem~\ref{theo:noglobaltracts} is now quite straightforward.
\begin{proof}[Proof of Theorem~\ref{theo:noglobaltracts}]
We prove the contrapositive, and so suppose that $f \in \A$ has a global tract. It is easy to see that $f$ can have no finite asymptotic values; this is an immediate consequence of, for example, \cite[Theorem 6]{MR0148923}. If $CV(f)$ was bounded, then we could deduce from Theorem~\ref{theo:conditionsfornoarctracts} part (\ref{t1first}) that $f$ has no arc tracts. Hence $CV(f)$ is unbounded, as required.
\end{proof}
%
%
%
\section{Proof of the second part of Theorem~\ref{theo:conditionsfornoarctracts}}
\label{S2}
This section is devoted to the proof of Theorem~\ref{theo:conditionsfornoarctracts} part (\ref{t1second}). In fact we prove a slightly stronger result. For any set $U \subset \mathbb{D}$, we let
$$
CV(f, U) = \{ f(z) : z \in U \text{ and } f'(z) = 0 \}.
$$

It is straightforward to see that Theorem~\ref{theo:conditionsfornoarctracts} part (\ref{t1second}) is an immediate consequence of the following.
\begin{theorem}
\label{theo:stronger}
Suppose that $f \in \A$ has an arc tract with end $K$, and that infinity is monotonically accessible at $\zeta_1$ and $\zeta_2$, which are interior points of $K$. Let $K'$ be the subarc of $K$ with endpoints $\zeta_1$ and $\zeta_2$, and suppose that $U \subset \mathbb{D}$ is a boundary neighbourhood of $K'$. Then $CV(f, U)$ is unbounded.
\end{theorem}

The \emph{level sets} of a complex valued function $f$ are defined, for each $r>0$, by
\begin{equation*}
L(r) = \{ z : |f(z)| = r \}.
\end{equation*}
The connected components of a level set are called \emph{level curves}. The proof of Theorem~\ref{theo:stronger} depends on certain properties of the level curves of a holomorphic function, given in the following lemma. This result may be known, but we have not been able to trace a reference.
\begin{lemma}
\label{lemm:components}
Suppose that $f : \D \to \C$ is holomorphic and unbounded. Suppose also that $R > 0$, and that $P(R)$ is a component of $L(R)$ compactly contained in $\D$ and meeting no critical points of $f$. Then there exists $\tilde{R} > R$ such that, for $r \in [R, \tilde{R})$, there exist components $P(r)$ of $L(r)$, all Jordan curves, such that the map $r \mapsto P(r)$ is continuous (in the Hausdorff metric) and $P(r')$ surrounds $P(r)$, for $r' \in (r, \tilde{R})$, and also such that at least one of the following occurs:
\begin{enumerate}[(a)]
\item the curves $P(r)$ accumulate at a critical point of $f$ as $r \rightarrow \tilde{R}$;
\item the curves $P(r)$ accumulate at a point of $\partial \D$ as $r \rightarrow \tilde{R}$.
\end{enumerate}
\end{lemma}
\begin{proof}

Let $\Omega$ be a neighbourhood of $P(R)$, compactly contained in $\D$, such that $f$ is locally univalent and $m$-to-$1$ on $\Omega$, for some $m \in\N$. Choose $z_0 \in \Omega$ such that $w_0 = f(z_0)$ satisfies $|w_0| > R$. Let $H = \{ t \in \C : \operatorname{Re } t > \log R \}$, and take $t_0 \in H$ such that $e^{t_0} = w_0$. Then there exists a vertical open strip $S_0$, containing $t_0$ and $L = \{ t \in \C : \operatorname{Re } t = \log R\}$, such that $\exp(S_0) \subset f(\Omega)$. It follows that the branch of the function $h(t) = f^{-1}(\exp(t))$ that maps $t_0$ to $z_0$ can be analytically continued throughout $S_0$.

Now we let $S$ be the maximal vertical open strip with $L \subset \partial S$ to which $h$ can be analytically continued. We claim first that $S \ne H$. For if $S = H$, then there is a domain $W \subset \D$ such that the map $g(z) = R/f(z)$ is an unbranched covering map from $W$ onto $\D\setminus\{0\}$. Hence we can apply Lemma~\ref{lemm:covering} to $g$, with $r = 0$. Note that Lemma~\ref{lemm:covering} case (a) cannot hold, since $f$ maps $P(R)$ in an $m$-to-$1$ manner onto $\{ w \in \C : |w| = R \}$. If Lemma~\ref{lemm:covering} case (b) holds, then $W$ is doubly connected and so $\partial W = P(R) \cup \partial \D$. In this case the function $g$, which is analytic in $W$, approaches zero at each point of $\partial \D$, and so can be continued analytically across $\partial \D$ by Schwarz reflection. It follows that $g$ is identically zero, by the identity theorem. This contradiction concludes the proof of our claim.

It follows that the maximal vertical strip $S$ is of the form 
$$
S = \{ t \in \C : \log R < \operatorname{Re } t < \log \tilde{R} \}, \qfor \text{some } \tilde{R} > R.
$$
We put $W = h(S)$. It can then be seen, by a second application of Lemma~\ref{lemm:covering}, that $f$ must be of the form $f = (\psi)^m$, where $\psi : W \to \{ w \in \C : R^{1/m} < |w| < \tilde{R}^{1/m} \}$ is univalent.

Now let $P(r) = h(\{ t \in \C : \operatorname{Re } t = \log r \})$, for $r \in (R, \tilde{R})$. Since the strip $S$ is maximal, the circle $\{ w \in \C : |w| = \tilde{R}\}$ must contain a singular value, $w'$, of $f$, associated with the failure of analytic continuation of the corresponding inverse branch of $f$ along a path in $\{w\in\C: R<|w|<\tilde{R}\}$. If $w'$ is a critical value of $f$, then the components $P(r)$, $r \in (R, \tilde{R})$, lying in $W$ tend to a critical point of $f$, with critical value $w'$, as $r \to \tilde{R}$. This is case (a). Otherwise, $w'$ is an asymptotic value of $f$ arising from a path tending to $\partial \D$ from within $W$, and it follows that the components $P(r)$, $r \in (R, \tilde{R})$, accumulate at at least one point of $\partial \D$ as $r \rightarrow \tilde{R}$. This is case (b). The remaining results of the lemma follow easily.
\end{proof}
%
%
\begin{proof}[Proof of Theorem~\ref{theo:stronger}]
Suppose that $f \in \A$ has an arc tract with end $K$, and that $\zeta_1$, $\zeta_2$ and $K'$ are as in the statement of the theorem. For each $i \in \{1, 2\}$, let $\Gamma_i$ be an asymptotic curve over infinity, ending at $\zeta_i$, such that $f(\Gamma_i)$ is monotonic. By way of contradiction, suppose that there exist a boundary neighbourhood $U$ of $K'$ and a real number $M>0$ such that 
\begin{equation}
\label{eq:constraint}
CV(f, U) \subset B_M.
\end{equation} 

Reducing $\Gamma_1$ and $\Gamma_2$ if necessary, we can assume that $\Gamma_1 \cup \Gamma_2 \subset U$. Let $C$ be a simple curve in $U$ that joins $\Gamma_1$ and $\Gamma_2$ so that $\Gamma_1 \cup C \cup \Gamma_2$ is a crosscut of $\D$, and let $D$ be the boundary neighbourhood bounded by $\Gamma_1 \cup C \cup \Gamma_2 \cup K'$. See Figure~\ref{fig1} for an illustration of these sets.

\begin{figure}
	\includegraphics[width=14cm,height=10cm]{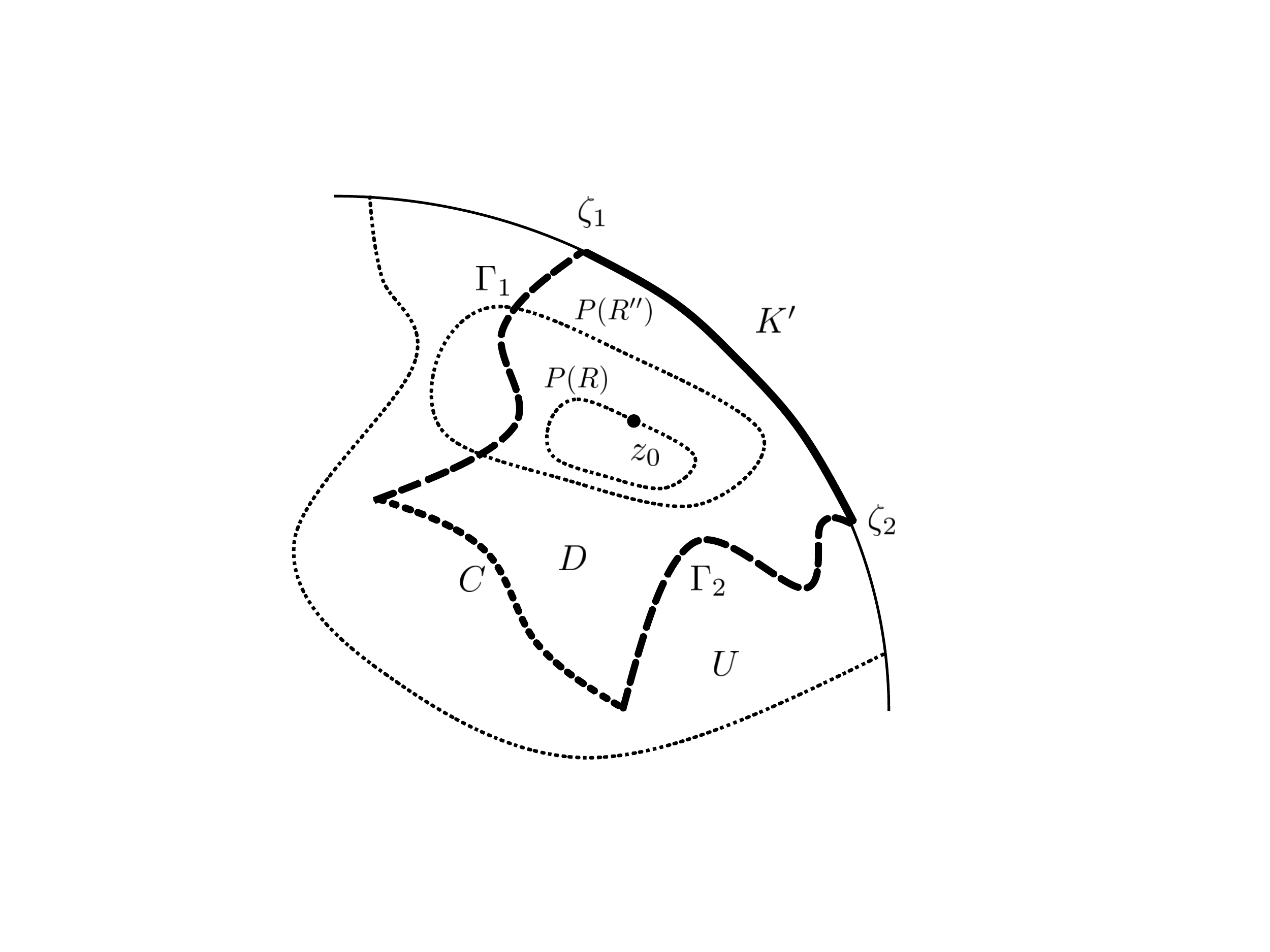}
  \caption{An illustration of the sets in the proof of Theorem~\ref{theo:stronger}}\label{fig1}
\end{figure}

Choose a real number $R > M$ sufficiently large that $|f(z)| < R$, for $z \in C$. We can also assume that $R$ has the property that $f(\Gamma_i)$ has exactly one point of modulus $r$, for $r > R$ and $i \in \{1, 2\}$. Since $f$ has only countably many critical values, we can also assume that $f$ has no critical values of modulus $R$. It follows that each component of $L(R)$ is either a Jordan curve or a curve tending to $\partial \D$ at both ends.

Fix a value $w_0 \in \C$ such that $|w_0| = R$. By Lemma~\ref{lemm:maclane} part \eqref{item:allvals}, there is a point $z_0 \in D$ such that $f(z_0) = w_0$. Let $T_0$ be the component of $L(R)$ containing $z_0$. We claim that we can assume that $T_0$ is contained in $D$. Note that $\overline{T_0}$ cannot meet $K'$ (see \cite[Remark after Theorem 3]{MR0148923}) since $K$ is the end of an arc tract over infinity. It will follow from this claim that $T_0$ is a Jordan curve.

To prove the claim, suppose that $T_0$ is not contained in $D$. Since $\overline{T_0}$ cannot meet $K'$, we can use Lemma~\ref{lemm:maclane} part \eqref{item:allvals} again to choose a point $z_1 \in D \setminus T_0$ such that $f(z_1) = w_0$. Let $T_1$ be the component of $L(R)$ containing $z_1$. If $T_1$ is also not contained in $D$, then we use Lemma~\ref{lemm:maclane} part \eqref{item:allvals} a third time to choose a point $z_2 \in D \setminus (T_0 \cup T_1)$ such that $f(z_2) = w_0$. Let $T_2$ be the component of $L(R)$ containing $z_2$. Suppose that $T_2$ is also not contained in $D$. 

We now have three disjoint curves in $\D$, each of which meets $D$ but is not contained in $D$, and so meets $\partial D$. Note that none of these curve can meet $C$, by the choice of $R$. 
It follows that each curve meets either $\Gamma_1$ or $\Gamma_2$. Without loss of generality, we can assume that $T_0$ and $T_1$ each meet $\Gamma_1$, and, moreover, at distinct points of $\Gamma_1$ (by \eqref{eq:constraint}). A contradiction then follows from the fact that $f(\Gamma_1)$ is monotonic. This proves our claim that we can assume that $T_0$ is a Jordan curve contained in $D$, since we can replace $z_0$ with $z_1$ or $z_2$ if necessary.

We complete the proof by showing that we can choose a different level curve which meets $\Gamma_i$ in two distinct points, for some $i \in \{1, 2\}$, which is a contradiction, since $f(\Gamma_i)$ is monotonic. Using the notation of Lemma~\ref{lemm:components}, we let $P(R) = T_0$. It follows by Lemma~\ref{lemm:components} that, for $r>R$ sufficiently close to $R$, we can let $r \mapsto P(r)$ be the continuous map to a component of $L(r)$. We claim that we can choose $r > R$ such that $P(r)$ is a Jordan curve which does not lie in $\overline{D}$ and meets either $\Gamma_1$ or $\Gamma_2$ in two points, which gives the required contradiction. We consider the two cases in the conclusion of Lemma~\ref{lemm:components}.
\begin{enumerate}[(a)]
\item Suppose that there is a value $\tilde{R} > R$ such that $P(r)$ accumulates at a critical point of $f$ as $r\rightarrow \tilde{R}$, and $P$ is continuous on $(R, \tilde{R})$. Since $\overline{D}\cap\D \subset U$, it follows from equation \eqref{eq:constraint} that this critical point must lie outside $\overline{D}$. The claim follows.
\item Suppose that there is a value $\tilde{R} > R$ such that $P(r)$ accumulates at a point of $\partial \D$ as $r \rightarrow \tilde{R}$, and $P$ is continuous on $(R, \tilde{R})$. Since $\overline{P(\tilde{R})}$ cannot meet the interior of $K$, the claim follows.
\end{enumerate}
This completes the proof.
\end{proof}
%
%
%
\section{Proof of Theorem~\ref{theo:moreconditions}}
\label{S4}
In this sections we prove Theorem~\ref{theo:moreconditions}. To do this we need to use MacLane's two alternative characterizations of class $\A$, and these require a little additional terminology. Suppose that $S \subset \D$. For each $\epsilon \in (0, 1)$, we let $\delta(\epsilon)$ denote the supremum of the diameters of the components of $S \cap \{ z \in \D : 1-\epsilon < |z| < 1\}$; if this intersection is empty, then we set $\delta(\epsilon) = 0$. We say that $S$ \emph{ends at points} if $\delta(\epsilon) \rightarrow 0$ as $\epsilon\rightarrow 1$. 

A non-constant holomorphic function~$f$ defined in $\D$ is said to be:
\begin{enumerate}[(a)]
\item in class $\B$ if there exists a set of boundary paths in~$\D$, with endpoints dense in $\partial \D$, on each of which the function $f$ is either bounded or has asymptotic value infinity;
\item in class $\LL$ if each level set $L=\{z\in\D:|f(z)|=\lambda\}$ ends at points. 
\end{enumerate}

%

MacLane \cite[Theorem~1]{MR0148923} showed that $\A=\B=\LL$. Using the fact that $\A~=~\LL$, MacLane and Hornblower each gave sufficient growth conditions for $f\in \mathcal A$. MacLane's sufficient condition \cite[Theorem~14]{MR0148923} is
$$
\int_0^1(1-r)\log^+ |f(re^{i\theta}|\,dr <\infty,\quad \text{for a dense set of }\theta\in(0,2\pi),
$$
which is best possible; see \cite[Theorem~1.1]{MR1956145}. Hornblower's sufficient condition \cite[Theorem~1]{MR0296305} is
\begin{equation}\label{Horn}
\int_0^1\log^+\log^+ M(r,f)\,dr<\infty,
\end{equation}
where
\[
M(r,f)=\max\{|f(z)|:|z|=r\}, \;\text{for } 0<r<1,
\]
which is essentially best possible; see \cite[Theorem 10.21]{MR1049148}.

%

Both parts of the proof of Theorem~\ref{theo:moreconditions} use the following lemma, which has its origins in an argument of MacLane \cite[page 284]{MR0274765}.
\begin{lemma}\label{parts}
Suppose that $f$ is holomorphic in $\D$ and that $\Gamma$ is a simple curve in $\D$ with endpoints $z_0$ and $z_1$, where we allow that $z_1 \in \partial \D$, and that $f'$ has no critical points in the interior of $\Gamma$. Suppose finally that there exists $c \in [0,1)$ such that $f'(\Gamma)$ is the ray with initial point $f'(z_0)$ and final point $cf'(z_0)$. Then
\begin{equation}\label{zn2}
|f(z_0)-f(z_1)|\le 2|f'(z_0)|,
\end{equation}
where, in the case when $z_1\in \partial \D$, the value $f(z_1)$ is taken to be the limiting value of $f(z)$ as $z \to z_1$ along $\Gamma$. 
\end{lemma}
\begin{proof}
Since $f'$ has no critical points on $\Gamma$, we can define a parametrization $\phi(t)$ of $\Gamma$ such that
\[
f'(\phi(t))=f'(z_0)(c+t(1-c)), \qfor 0\le t\le 1,
\]
so $\phi(1) = z_0$ and $\phi(0) = z_1$. Suppose first that $z_1\in \D$. It follows that
\begin{align*}
f(z_0)-f(z_1)&=\int_0^1f'(\phi(t))\phi'(t)\,dt\\
&=f'(z_0)\int_0^1(c+t(1-c))\phi'(t)\,dt\\
&=f'(z_0)\left(z_0-cz_1-(1-c)\int_0^1\phi(t)\,dt\right),
\end{align*}
where in the final step we have integrated by parts. Equation \eqref{zn2} now follows easily. The case when $z_1\in \partial \D$ follows by a straightforward limiting process.
\end{proof}
\begin{proof}[Proof of Theorem~\ref{theo:moreconditions}]
Suppose that~$f$ is holomorphic in $\D$ and that $CV(f)$ is bounded. To prove part (a) we show that if $f\in{\mathcal A}$ and $f$ has no arc tracts, then $f'\in \A$. If, on the contrary, $f'\notin \A$, then there is a non-degenerate closed subarc~$\gamma$ of $\partial \D$ such that $f'$ does not have an asymptotic value at any point of~$\gamma$. We show that this assumption leads to a contradiction.

We claim first that we can choose a sequence $(z_n)_{n\in\N}$ in $\D$ such that $z_n\to \zeta$ as $n\to\infty$, where~$\zeta$ is an interior point of~$\gamma$, and such that
\begin{equation}\label{zn}
f(z_n)\to \infty\;\;\text{as }n\to\infty\quad\text{and \quad the sequence } (f'(z_n))_{n\in\N} \text{ is bounded.}
\end{equation}

We prove \eqref{zn} as follows. Suppose that $U \subset \D$ is any Jordan domain such that $\partial U \cap \partial \D = \gamma$. Then $f'$ is not in class $\A$ relative to $U$. Since $\A = \LL$, it follows that $f'$ is not in class $\LL$ relative to $U$. Since a level set of $f'$ can only accumulate in $\gamma$, it follows that there exist $\lambda > 0$ and a non-degenerate closed subarc~$\gamma'$ of the interior of $\gamma$ at which a level set $$L=\{z\in\D: |f'(z)|=\lambda\}$$ accumulates. 

It follows that, without loss of generality, we can take a sequence of curves $L_n\subset L$ such that $L_n \to \gamma'$ as $n\to\infty$; see \cite[page~10]{MR0148923}. Suppose, by way of contradiction, that $f$ is bounded on $\cup_{n=1}^{\infty}L_n$. Then $f$ is bounded on a boundary neighbourhood of each interior point of~$\gamma'$, since $f\in \A$. Hence $f'$ does not grow too quickly near interior points of~$\gamma'$, by Cauchy's estimate; in particular there is a bound on $|f'(z)|$ of the form $O((1-|z|)^{-1})$ as $|z|\to 1$ there. It follows by \eqref{Horn}, applied locally, that $f'$ must have asymptotic values at a dense subset of~$\gamma'$. This contradicts our assumption about~$\gamma$. Hence, $f$ is unbounded on $\cup_{n=1}^{\infty}L_n$, and so property \eqref{zn} follows.

Suppose that $n \in \N$. Without loss of generality we can assume that there are no critical values of $f'$ on the radial ray with initial point $f'(z_n)$ and final point $0$ (except possibly at $0$). We let $c_n\in [0,1)$ be the smallest value such that there is a simple curve $\Gamma_n \subset \D$, with one endpoint at $z_n$, such that $f'(\Gamma_n)$ is the radial ray from $f'(z_n)$ to $c_nf'(z_n)$. 


By Lemma~\ref{parts}, we have that
\begin{equation}\label{zn3}
|f(z_n)-f(z)|\le 2|f'(z_n)|,\qfor z\in \Gamma_n.
\end{equation}
Thus, by \eqref{zn} and the fact that $CV(f)$ is bounded, we can assume that $\Gamma_n$ does not end at a critical point of $f$. We deduce that $\Gamma_n$ accumulates at $\partial \D$.

We claim that none of the curves $\Gamma_n$ can accumulate on a non-degenerate subarc of~$\gamma$. Suppose, by way of contradiction, that $\gamma_n$ is such a subarc. Note that $f$ is bounded on each $\Gamma_n$, by \eqref{zn3}. Since $f\in\A$, we can deduce by the maximum principle that~$f$ is bounded in a boundary neighbourhood of each interior point of $\gamma_n$. Thus, by Cauchy's estimate again and \eqref{Horn}, the function $f'$ has asymptotic values at a dense subset of $\gamma_n$, contrary to our original assumption. This proves our claim.

Since $f'$ has no asymptotic values on $\gamma$, it follows that each~$\Gamma_n$ must accumulate at at least one point of $\partial \D$ lying outside the interior of~$\gamma$. Since $z_n\to\zeta$ as $n\to\infty$, where $\zeta$ is an interior point of~$\gamma$, we can assume, by taking a subsequence if necessary, that $\Gamma_n\to \Gamma$ as $n\to\infty$, where $\Gamma$ is a non-degenerate subarc of $\partial \D$ which has non-degenerate intersection, $\gamma''$ say, with~$\gamma$. Hence, by \eqref{zn} and \eqref{zn3}, the function~$f$ must have an arc tract for infinity with end including $\gamma'' $ (see \cite[Theorem~3]{MR0148923}). This contradicts our hypothesis about~$f$. \\

Next we show that if $f'\in{\mathcal A}$ and $f'$ has no arc tracts, then $f\in \A$ and~$f$ has no arc tracts. Suppose, for a contradiction, that $f\notin \A$. Then, since $\A=\B$, there is a non-degenerate arc $\gamma\subset \partial \D$ each point of which is {\em not} the endpoint of any path in $\D$ on which $f$ is either bounded or tends to infinity. Once again, we claim that we can choose a sequence $(z_n)_{n\in\N}$ in $\D$ such that $z_n\to \zeta$, where $\zeta$ is an interior point of~$\gamma$, and such that
\begin{equation}\label{zn4}
f(z_n)\to \infty\;\;\text{as }n\to\infty\quad\text{and \quad the sequence } (f'(z_n))_{n\in\N} \text{ is bounded.}
\end{equation}
Indeed, since $f'\in\A$, $f'$ has no arc tracts, and $\A = \B$, there must be a non-degenerate subarc~$\gamma'$ of~$\gamma$ on which at least one of the following occurs:
\begin{enumerate}[(a)]
\item there exist paths approaching points of a dense subset of~$\gamma'$ on each of which $f'$ is bounded;
\item $f'$ has distinct point tracts for infinity ending at points of a dense subset of~$\gamma'$; for each such point of $\gamma'$ the components of $$\{z\in\D: r < |z| < 1, \ |f'(z)|>L \}$$ that define the tract must have diameters that tend to $0$ as $ L \to \infty$, so there exist level curves of~$f'$ ending at a dense set of points of~$\gamma'$.
\end{enumerate}
In either case, \eqref{zn4} holds, because~$f$ is not bounded on any path ending at a point of~$\gamma'$.

Now let $(\Gamma_n)_{n\in\N}$ be the sequence of paths, with initial points $(z_n)_{n\in\N}$, introduced in the first part of the proof. Once again, by \eqref{zn4}, we can assume that each path $\Gamma_n$ does not end at a critical point of~$f$, so it must accumulate at $\partial \D$. In fact, each $\Gamma_n$ must end at a point $\zeta_n\in\partial \D$ for otherwise $f'$ would have an arc tract for the value $c_nf'(z_n)$ (the asymptotic value of $f'$ along $\Gamma_n$), which is impossible by our initial hypothesis (or alternatively by Lemma~\ref{lemm:maclane} part \eqref{item:overinf}). By our assumption about $\gamma$ and \eqref{zn3}, the points $\zeta_n$ cannot lie in~$\gamma$.

Once again we can assume that $\Gamma_n\to \Gamma$ as $n\to\infty$, where $\Gamma$ is a non-degenerate subarc of $\partial \D$. Since $f'\in \A$ and $f'$ is uniformly bounded on $\cup_{n=0}^{\infty}\Gamma_n$, by \eqref{zn4}, we deduce that $f'$ is bounded in a boundary neighbourhood of each interior point of $\Gamma$, so~$f$ has asymptotic values at points of a dense subset of $\Gamma$, and hence at some point of $\gamma$ (since $\zeta$ lies in the interior of $\gamma$), a contradiction. Hence $f\in\A$.

To complete the proof of Theorem~\ref{theo:moreconditions}, we note that if~$f$ has an arc tract for infinity with end~$\gamma$, then for each path $\Gamma$ in $\D$ with an endpoint $\zeta$ in the interior of~$\gamma$, we have $\limsup_{z\to\zeta, z\in\Gamma}|f(z)|=\infty$; see \cite[Remark after Theorem~3]{MR0148923}. This fact, taken together with the existence of a non-degenerate subarc $\gamma'$ of $\gamma$ on which one of the two possible cases for $f'$ listed above occurs, indicates that \eqref{zn4} must again hold. We can then obtain the required contradiction as in the previous two paragraphs.
\end{proof}
%
%
%
\section{Proof of Theorem~\ref{theo:finBnotA}}
\label{S3}
In the final two sections we prove the following result.
\begin{theorem}
\label{theo:fspiral}
There is an unbounded function $f \in \Bdisc$, which is bounded on a spiral $\Gamma \subset \D$ that accumulates on the whole of $\partial \D$.
\end{theorem}
As noted in the introduction, since it follows from \cite[Theorem 9]{MR0148923} that $f \notin \A$, Theorem~\ref{theo:finBnotA} follows immediately from Theorem~\ref{theo:fspiral}.

The proof of Theorem~\ref{theo:fspiral} is split into two parts. First, in the remainder of this section, we define a $\D$--model, and then state a theorem regarding $\D$--models that is analogous to a construction of Bishop \cite[Theorem 1.1]{MR3384512}. We also show how to use this result to prove Theorem~\ref{theo:fspiral}. Then, in the final section, we give the proof of the result. Wherever possible we retain the notation of \cite{MR3384512}, for ease of reference.

We begin by defining a $\D$--model. Recall that $\mathbb{H}$ denotes the right half-plane $\mathbb{H} = \{ z \in \C  : \operatorname{Re}(z) > 0\}$. Suppose that $I$ is an index set which is at most countably infinite. Suppose that $\Omega = \bigcup_{j \in I} \Omega_j$ is a disjoint (possibly finite) union of domains $\Omega_j \subset \D$ that are simply connected and not compactly contained in $\D$. Suppose also that, for each $j\in I$, the map $\tau_j : \Omega_j \to \mathbb{H}$ is conformal. Let $\tau$ be the map $\tau : \Omega \to \C$ which is equal to $\tau_j$ on $\Omega_j$, for $j \in I$. Suppose that the following conditions all hold:
\begin{enumerate}[(a)]
\item Sequences of components of $\Omega$ accumulate only on $\partial \D$.\label{condaccum}
\item The boundary of $\Omega_j$ is connected, for $j \in I$.\label{condboundary}
\item For any sequence $(z_n)_{n\in\N}$ of points of $\Omega$, we have $|z_n|\rightarrow 1$ as $n\rightarrow\infty$ if and only if $\tau(z_n)\rightarrow\infty$ as $n\rightarrow\infty$.\label{condtoinf}
\end{enumerate}
Finally set $F = \exp \circ \tau$. Note that $F$ is a covering map from each $\Omega_j$ onto $\C\setminus\overline{\D}$. The pair $(\Omega, F)$ is called a \emph{$\D$--model}.

\begin{remark}\normalfont
We note that Bishop's definition of a \emph{model} is as above, but with $\D$ replaced by $\C$ and other modifications; in particular condition \eqref{condtoinf} is replaced by the weaker condition that if $(z_n)_{n\in\N}$ is a sequence of points of $\Omega$ such that $\tau(z_n)\rightarrow\infty$ as $n\rightarrow\infty$, then $z_n\rightarrow\infty$ as $n\rightarrow\infty$. Our stronger condition is needed to ensure that, in the disc setting, the function constructed has a bounded set of finite asymptotic values. Finally, we remark that we use the term $\D$--model mainly in order to distinguish between our setting and Bishop's.
\end{remark}

We now state our version of \cite[Theorem 1.1]{MR3384512}. For completeness we have retained parts of the result which are not required in our application. Here we let 
$$
\Omega(\rho) = \{ z \in \Omega : |F(z)| > e^\rho \}, \qfor \rho > 0,
$$
and
$$
\Omega(\rho_1, \rho_2) = \{ z \in \Omega : e^{\rho_2} > |F(z)| > e^{\rho_1} \}, \qfor \rho_2> \rho_1 > 0. 
$$
\begin{theorem}
\label{theo:construction}
Suppose that $(\Omega, F)$ is a $\D$--model, and that $\rho \in (0, 1]$. Then there exists an unbounded function $f\in\Bdisc$ and a quasiconformal map $\phi : \D \to \D$ such that the following all hold.
\begin{enumerate}
\item We have $f(\phi(z))=F(z)$, for $z \in \Omega(2\rho)$.\label{l0}
\item We have $|f(\phi(z))| \leq e^\rho$, for $z \notin \Omega(\rho)$ and $|f(\phi(z))| \leq e^{2\rho}$, for $z \notin \Omega(2\rho)$.\label{l2}
\item The quasiconstant of $\phi$ is $O(\rho^{-2})$ as $\rho\rightarrow 0$, with constant independent of $F$ and $\Omega$.
\item The map $\phi$ is conformal outside the set $\Omega(\rho/2, 2\rho)$.
\end{enumerate}
\end{theorem}

Finally in this section, we show how Theorem~\ref{theo:fspiral} can be deduced from Theorem~\ref{theo:construction}.
\begin{proof}[Proof of Theorem~\ref{theo:fspiral}]
Let $S \subset \D$ be a spiral which accumulates on the whole boundary of $\D$. For example, we can take $$S = \left\{ r e^{i\theta} : r \in [0,1), \ \theta = (1-r)^{-1}\right\}.$$

Let $\Omega = \Omega_1 = \D \setminus S$, and let $\tau = \tau_1$ be a conformal (Riemann) map from $\Omega$ to $\mathbb{H}$. Note that $\partial \Omega = \partial \D \cup S$. All points of $S$ are accessible boundary points of $\Omega$, and the whole of $\partial \D$ corresponds to a single prime end $E$ of $\Omega$. By choosing $\tau$ so that $E$ corresponds to infinity under $\tau$, we have that $(\Omega, \exp \circ \tau)$ is a $\D$--model. 

Let $f$ and $\phi$ be the functions that result from an application of Theorem~\ref{theo:construction} with $\rho = 1$. Clearly $f \in \Bdisc$, and $f$ is unbounded. Set $\Gamma = \phi(\partial \Omega(1))$. Then $f$ is bounded on $\Gamma$.

It remains to show that $\Gamma$ accumulates on the whole of $\partial \D$. The domain $\Omega(1)$ is not compactly contained in $\D$ and does not meet $S$. It follows that $\partial \Omega(1)$ accumulates on the whole of $\partial\D$. Since $\phi$ extends homeomorphically to $\partial\D$ (see, for example, \cite[Theorem 8.2]{MR0344463}), it follows that $\Gamma = \phi(\partial \Omega(1))$ accumulates on the whole of $\partial\D$. Hence $f$ has the properties required.
\end{proof}
%
%
%
\section{Proof of Theorem~\ref{theo:construction}}
\label{S5}
In this section we give the proof of Theorem~\ref{theo:construction}. In fact, the proof is extremely close to the original proof of Bishop's result in \cite[Theorem 1.1]{MR3384512}. Accordingly we only outline the proof, highlighting the differences. \\

First we note that we can assume that $\rho = 1$. Bishop's proof of this fact (even though we do not need it) applies immediately. For simplicity we also assume that $I = \{1\}$ is a singleton, and we replace $\Omega_1$ and $\tau_1$, etc., by $\Omega$ and $\tau$, etc. There is no particular additional difficulty in proving the result in the case that $I$ is countably infinite, apart from more complexity of notation.

Let $W = \mathbb{D}\setminus\overline{\Omega(1)}$ and let $\gamma = \partial W$. Let $L_1$ and $L_2$ be the vertical lines
$$
L_p = \{ z = p + iy : y \in \R \}, \qfor p \in \{1, 2\}.
$$
Note that $L_1 = \tau(\gamma)$. Then $W$ is an open, connected, simply connected domain, which is bounded by an analytic arc $\gamma$, which tends to $\partial \D$ in both directions; these properties follow from the three conditions \eqref{condaccum}, \eqref{condboundary} and \eqref{condtoinf} in the definition of a $\D$--model.

In particular, since $W$ is simply connected, we can let $\Psi$ be a Riemann map from $W$ to $\mathbb{D}$. Bishop shows that $\Psi$ can, in fact, be extended to $\D\setminus\overline{\Omega(2)}$, by Schwarz reflection, and so, in particular, $\Psi$ is defined in a neighbourhood of $\gamma$. \\ 

Roughly speaking, our goal is to construct a quasiregular map $g : \D \to \C$ with the right properties for $f$, and then use the measurable Riemann mapping theorem \cite{MR0115006} to recover $f$ itself. A ``first attempt'' to define $g$ might be as follows. First set $g(z) = F(z) = \exp(\tau(z))$, for $z \in \Omega(2)$, so that $g$ maps $\Omega(2)$ analytically to the exterior of the ball $\overline{B_{e^2}}$. Then set $g(z) = e \Psi(z)$, for $z \in W$, so that $g$ maps $W$ conformally to $B_e$. Then interpolate, somehow, in the ``strip'' $\Omega(1) \setminus \overline{\Omega(2)}$. 

Unfortunately there is no reason to believe that $g|_{\gamma}$ is sufficiently close to $g|_{\partial\Omega(2)}$ to enable this interpolation to be constructed in a way which is quasiregular.  

We need something more complicated. The construction works in two stages. Roughly speaking the first stage replaces $e\Psi$ with a map which takes values on $\gamma$ comparable to those of $g$ on $\partial \Omega(2)$. The second stage is the interpolation, and this happens in $\{ z = x + iy : 1 \leq x \leq 2, \ y \in \R\}$. \\

The first stage is as follows. Let $B$ be a certain Blaschke product multiplied by $e$, so that $B : \D \to B_e$. Consider also the sets 
$$
\mathcal{J} = \{ z \in L_1 : \exp(z) = e \}  = \{ 1 + 2n\pi i : n \in \Z\},
$$
and
$$
\mathcal{L} = \{ z \in L_1 : (B\circ\Psi\circ\tau^{-1})(z) = e \}.
$$

We want these two sets to be comparable in the following sense. The set $\mathcal{J}$ partitions $L_1$ into intervals which we call $J$-intervals. Similarly, $\mathcal{L}$ partitions $L_1$ into intervals which we call $L$-intervals. Bishop shows that we can choose the map $B$ so that the following holds. There is an integer $M$ such that each $L$-interval meets at least two $J$-intervals and at most $M$ $J$-intervals. It follows that no $J$-interval contains an $L$-interval.

The technique for proving the existence of such a $B$ is as follows. First the points of $\mathcal{J}$ are pulled back to the boundary of $\D$ by the map $\Psi \circ \tau^{-1}$; let these points be denoted by $(a_n)_{n\in\N}$. The map $B$ is then defined by constructing a certain infinite set $\mathcal{K} \subset \N$, and defining
$$
B(z) = e \prod_{k \in \mathcal{K}} \frac{|a_k|}{a_k}\frac{a_k - z}{1 - \overline{a_k}z}.
$$
The fact that there exists a set $\mathcal{K}$ such that $B$ has the required properties is proved using two sets of facts; see \cite[Section 4]{MR3384512}. The first is that images of adjacent $J$-intervals under $\Psi \circ \tau^{-1}$ have comparable size (with uniform bounds), which follows from the Koebe distortion theorem. The second concerns certain properties of the harmonic measure of boundary arcs of the unit disc $\Psi(W)$. These facts all carry across from Bishop's proof into our setting directly, and we omit further detail. \\

The second stage is the actual interpolation, and is quite complicated. However, it should be noted that this part of the construction happens entirely within the closed strip
$$
S = \{ z = x + iy : 1 \leq x \leq 2, \ y \in \R \},
$$
and only uses the established relationship between $\mathcal{J}$ and $\mathcal{L}$. So this entire stage carries over from Bishop's setting completely unchanged.

\begin{figure}
	\includegraphics[width=14cm,height=10cm]{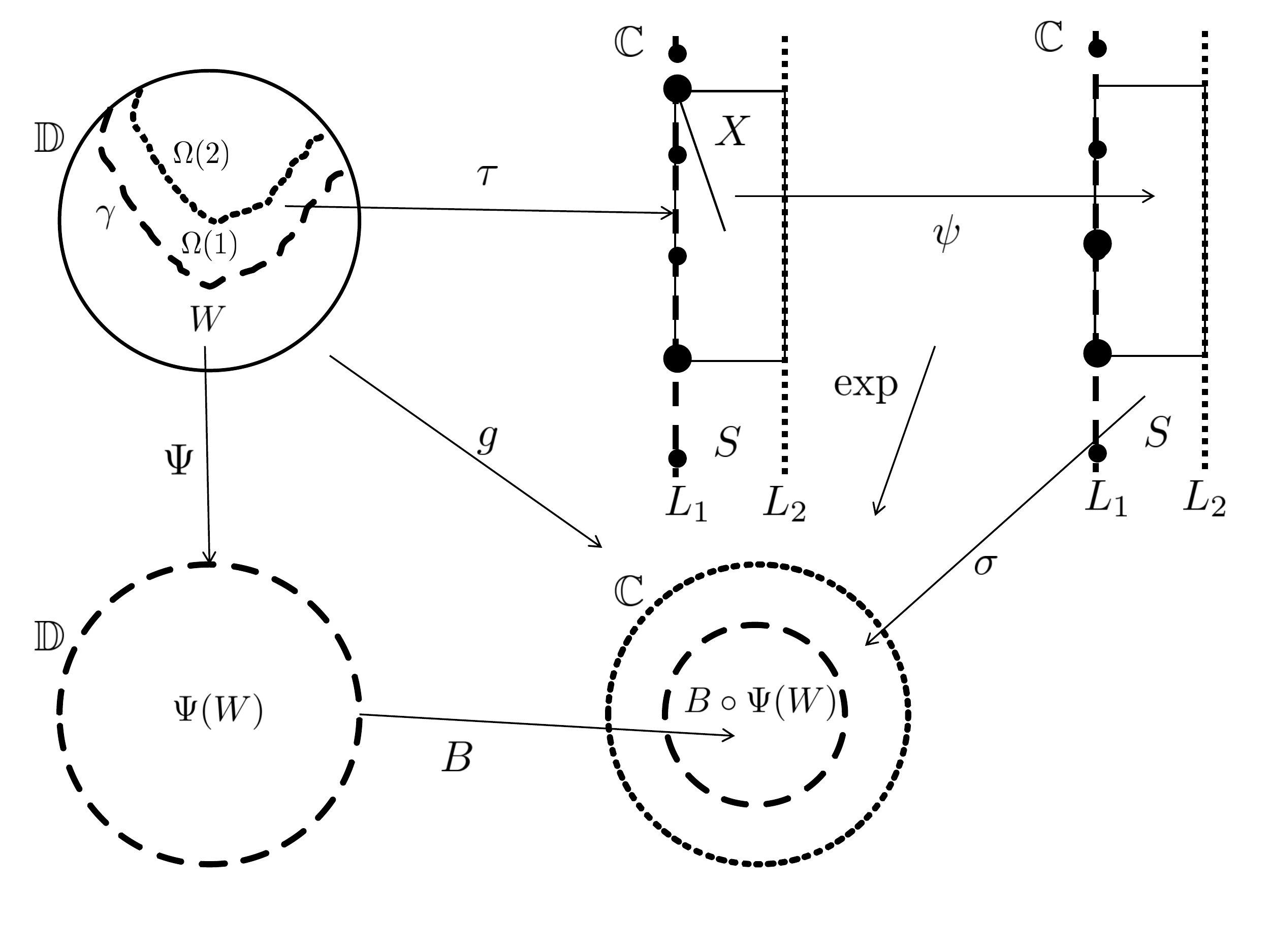}
  \caption{The full interpolation between $B \circ \Psi$ and $F$. On $L_1$, elements of $\mathcal{J}$ are shown as small dots, and two elements of $\mathcal{L}$ are shown as large dots. After the function $\psi$ these elements of $\mathcal{L}$ are coincident with elements of $\mathcal{J}$.}\label{fig2}
\end{figure}

For completeness, though, we outline the construction. The set $S$ is first divided into a collection of rectangles aligned with the axes, so that two corners of each rectangle are adjacent points of $\mathcal{L}$. One such rectangle is shown in Figure~\ref{fig2}. The rectangle is equipped with a slit running part-way along a diagonal (see Figure~\ref{fig2} again), and then four quasiregular maps are applied to this set (we will call it $X$).
\begin{enumerate}[(a)]
\item The first map, $\psi_1$, is the identity on all sides of $X$ apart from the left-hand side. Note that this left-hand side is mapped by $B \circ \Psi \circ \tau^{-1}$ to a circle. The effect of $\psi_1$ on this side is to ``stretch'' it so that -- for example -- the point which is mapped by $B \circ \Psi \circ \tau^{-1}$ half-way around the circle is mapped by $\psi_1$ to a point half-way up the side of $X$. (Bishop terms $\psi_1$ a \emph{straightening}.)
\item The second map, $\psi_2$, linearly moves and stretches the rectangle to align it with elements of $\mathcal{J}$.
\item The third map, $\psi_3$, which is the most complicated, opens out the slit, and linearly maps the element of $\mathcal{L}$ that was the left-hand side of $X$ to an element of $\mathcal{J}$ (say $J$). The composition $\psi_3 \circ \psi_2 \circ \psi_1$ is denoted by $\psi$; see Figure~\ref{fig2}. Note that $\psi$ is not even continuous (or even defined) on the whole of $S$. (In Bishop's terms, $\psi_3$ is a \emph{quasiconformal folding}.)
\item The final map (denoted by ${\sigma}$, and see Figure~\ref{fig2} again) is carefully constructed to resolve the continuity issue mentioned in the previous sentence, and is equal to the exponential function on $J$.
\end{enumerate}

Using these maps we can define a quasiregular map $g : \D \to \C$ by
$$
g(z) = 
\begin{cases}
(B \circ \Psi)(z), &\text{for } z \in W, \\
(\exp \circ \ \tau)(z), &\text{for } z \in \Omega(2), \\
(\sigma \circ \psi \circ \tau)(z), &\text{for } z \in \overline{\Omega(2) \setminus \Omega(1)}.
\end{cases}
$$

Since $g$ is a quasiregular map from $\D$ to $\C$, it follows from the measurable Riemann mapping theorem that there exists a quasiconformal map $\phi : \D \to \D$ such that $f = g \circ \phi^{-1}$ is a holomorphic map from $\D$ to $\C$. 

We need to show that $f$ has the required properties. The fact that we have $f(\phi(z)) = F(z)$, for $z \in \Omega(2)$, is immediate. The facts that $|f(\phi(z))| \leq e$, for $z \notin \Omega(1)$ and $|f(\phi(z))| \leq e^{2}$, for $z \notin \Omega(2)$ are also immediate. The last two conclusions of the theorem also follow, though we do not require these and the proof is omitted.

It remains to show that the singular values of $f$ are bounded, and so $f \in \mathcal{B}_{\D}$. Note that the singular values of $f$ and $g$ coincide. It is easy to see that $g$ has no critical points in $\Omega(2)$. It follows that the critical values of $f$ are bounded. If $g$ has a finite asymptotic value of modulus greater than $e^2$, then this value must be the limit of $g$ along a curve $\Gamma$ lying in $\Omega(2)$ which tends to $\partial \D$. In this case $e^z$ has a finite limit along $\tau(\Gamma) \subset \mathbb{H}$. Moreover, it follows from the final property of a $\D\text{--model}$ that $\tau(\Gamma)$ is unbounded, which is in contradiction to the previous sentence. It follows that the asymptotic values of $f$ are also bounded, and this completes the proof.
%
%
%
%
%
%
%
%
%

\end{document}